\documentclass{amsart}
\usepackage{graphicx}
\usepackage{latexsym}
\usepackage{amsfonts}
\usepackage[all]{xy}
\usepackage{amssymb, mathrsfs, amsfonts, amsmath}
\usepackage{amsbsy}
\usepackage{amsfonts}
\newtheorem{lemma}{Lemma}

\newtheorem{remark}{Remark}
\newtheorem{theorem}{Theorem}
\newtheorem{example}{Example}
\numberwithin{equation}{section}

\begin{document}
\title[quasi-Einstein metrics on $\Bbb{H}^{n}\times \Bbb{R}$]{A note on the uniqueness\\ of quasi-Einstein metrics on $\Bbb{H}^{n}\times \Bbb{R}$}
\author{E. Ribeiro Jr$^{1}$}
\address{$^{1}$ Departamento de Matem\'{a}tica-Universidade Federal do Cear\'{a}\\
60455-760-Fortaleza-CE-BR} \email{ernani@mat.ufc.br}
\thanks{$^{1}$ Partially supported by FUNCAP-Brazil}
\author{K. Bezerra $^{2}$}
\address{$^{2}$Current: Departamento de Matem\'{a}tica-Universidade Federal do Cear\'{a},
60455-760-Fortaleza-CE-BR\\
Permanent: Departamento de Matem\'{a}tica-Universidade Federal do Piau\'i, 64049-550-Teresina-PI-BR}\email{kelton@ufpi.edu.br}
\thanks{$^{2}$ Partially supported by CAPES-Brazil}
 \urladdr{http://www.mat.ufc.br}
\date{July 15, 2013}

\begin{abstract}
The aim of this note is to give an explicit description of
quasi-Einstein metrics on $\Bbb{H}^{n}\times \Bbb{R}.$ We
shall construct two examples of quasi-Einstein metrics on this manifold and then we shall prove the uniqueness of these examples.
Finally, we clarify the relationship between quasi-Einstein metrics and
static metrics in the quoted space.
\end{abstract}

\maketitle

\section{Introduction}
In the last years very much attention has been given to Einstein
metrics and its generalizations, for instance Ricci solitons and
quasi-Einstein metrics. Ricci solitons model the formation of
singularities in the Ricci flow and  they correspond to self-similar
solutions, i.e. they are stationary points of this flow in the space
of metrics modulo diffeomorphisms and scalings; for more details in
this subject we recommend  the survey due to Cao \cite{Cao}  and the
references therein. On the other hand, one of the motivation to
study quasi-Einstein metrics on a Riemannian manifold is its closed
relation with  warped product Einstein metrics, see e.g.
\cite{besse}, \cite{csw}, \cite{mastrolia}, \cite{rimoldi}, \cite{hepeterwylie} and \cite{kk}.

In \cite{besse} it was proposed to find new examples of Einstein metrics on warped products. The authors wrote:
\begin{flushright}
\begin{minipage}[t]{4.37in}
 \emph{``Nevertheless warped products do give new examples of complete Einstein manifolds and the Einstein equations are quite interesting".}[see chapter 9 page 265.]
 \end{minipage}
\end{flushright}

Based on this problem we shall give explicit details of how to
solve this question by using quasi-Einstein theory.  One
fundamental ingredient to understand the behavior of such a class of
manifolds is the $m$-Bakry-Emery Ricci tensor which appeared
previously in \cite{qian} and \cite{bakry}. It is given by
\begin{equation}
\label{bertens}
Ric_{f}^{m}=Ric+\nabla ^2f-\frac{1}{m}df\otimes df,
\end{equation}where $f$ is a smooth function on $M^n$ and $\nabla ^2f$ stands for the Hessian form.

This tensor was extended recently, independently, by Barros and Ribeiro Jr \cite{brG} and Limoncu \cite{limoncu}. More precisely, they extended (\ref{bertens}) for an arbitrary vector field $X$ on $M^n$ as follows:
\begin{equation}
\label{eqprincqem}
Ric_{X}^{m}=Ric+\frac{1}{2}\mathcal{L}_{X}{g}-\frac{1}{m}X^{\flat}\otimes X^{\flat},
\end{equation}where $\mathcal{L}_{X}{g}$ and $X^{\flat}$ denote, respectively, the Lie derivative on $M^n$ and the canonical $1$-form associated to $X.$

Letting $X\in\mathfrak{X}(M),$ in this more general setting, we say that $(M^n,\,g)$ is a quasi-Einstein metric, if there exist constants $0<m\leq\infty$ and $\lambda$ such that
\begin{equation}
\label{eqprincqem1}
Ric_{X}^{m}=\lambda g.
\end{equation} The metric $g$ will be called a quasi-Einstein metric.

On the other hand, when $m$ goes to infinity, equation
(\ref{eqprincqem1}) reduces to the one associated to a Ricci soliton.
Whereas, when $m$ is a positive integer and $X$ is gradient.  Following the terminology used for Ricci
solitons, a  quasi-Einstein structure $g$ on a manifold $M^n$ will be
called \emph{expanding}, \emph{steady} or \emph{shrinking},
respectively, if  $\lambda<0,\,\lambda=0$ or $\lambda>0$. Moreover,
a quasi-Einstein structure will be called \emph{trivial} if $
X\equiv0$. Otherwise, it will be \emph{nontrivial}. We notice that
the triviality implies that $M^n$ is an Einstein manifold.

Classically the study of such metrics is considered when $X$ is the
gradient of a smooth function $f$ on $M^n.$ This will be the case
considered in this work. From now on, when quoting the
quasi-Einstein manifold we will refer to the gradient case.
Therefore, a Riemannian manifold $(M^n,\,g),$ $n\geq 2,$ will be called
quasi-Einstein metric if there exist a smooth potential function $f$
on $M^n$ and a constant $\lambda$ satisfying the following
fundamental equation
\begin{equation}
\label{eqqem}
Ric_{f}^{m}=Ric+\nabla ^2f-\frac{1}{m}df\otimes
df=\lambda g.
\end{equation}

In order to proceed we remember that on a compact manifold $M^n$ an
$\infty-$quasi-Einstein metric (Ricci soliton) with $\lambda\leq0$
is trivial, see \cite{monte}. The same result was proved previously
in \cite{kk}  for quasi-Einstein metric on compact manifold with $m$
finite. Besides, we known that compact shrinking Ricci solitons have
positive scalar curvature, see e.g. \cite{monte}. An extension of
this result for shrinking quasi-Einstein metrics with $1\leq
m<\infty$ was obtained in \cite{csw}. Some results on scalar curvature estimates for quasi-Einstein manifolds appears in \cite{rimoldi}. Recently, in
\cite{brozosGarciagavino} Brozos-V\'{a}zquez et al. proved that
locally conformally flat quasi-Einstein metrics are globally
conformally equivalent to a space form or locally isometric to a
$pp$-wave or a warped product. In \cite{hepeterwylie} was given some
classification for quasi-Einstein metrics where the base has non
empty boundary. Moreover, they proved a characterization for
quasi-Einstein metrics is locally conformally flat.

We point out that Bakry and Ledoux \cite{bakry} proved an analogue
of Myers's theorem and also they presented a new analytic proof of
Cheng's theorem based on Sobolev inequalities. Bakry and Ledoux's
result implies that every shrinking quasi-Einstein metric must be
compact. Moreover, Case proved nonexistence of steady
quasi-Einstein metric with $\mu\leq 0,$ where $\mu$ is a constant satisfying $\Delta_{f}f=-m\mu e^{\frac{2}{m}f},$ save for the trivial ones, for more details see \cite{case}.
Combining Bakry-Ledoux's result with Case's theorem we conclude that
every nontrivial noncompact quasi-Einstein metrics are expanding
provided $\mu\leq 0.$ For instance, it is well-known that
$\Bbb{H}^n$ with its canonical metric admits a nontrivial expanding
quasi-Einstein structure. Moreover, it is important to highlight that a
Euclidean space $\Bbb{R}^n$ and a Euclidean sphere $\Bbb{S}^n$ do
not admit a nontrivial quasi-Einstein structure.

In  \cite{bairdexplicit}, Baird has proved an explicit
constructing of Ricci solitons structures on $Nil^4$ by using ODE
technics. Inspired on ideas developed in Baird's work we shall
describe explicitly the quasi-Einstein structures on  $\Bbb{H}^{n}\times \Bbb{R}.$

From now on, we shall fix the standard metric on $\Bbb{H}^{n}\times \Bbb{R}$ which is given by
\begin{equation}
\label{metHR}
 g=\frac{1}{x_{n}^{2}}\sum_{i=1}^{n}dx_{i}^{2}+dt^{2}.
\end{equation}
We shall prove that $\big(\Bbb{H}^{n}\times \Bbb{R},g \big)$ admits only two quasi-Einstein structures. Our first example will be
obtained with a Killing vector field. More precisely, we have the
following example.
\begin{example}
\label{ex1} We consider $\mathbb{H}^{n}\times\mathbb{R}$ with
standard metric (\ref{metHR}) and the potential function
$f(x,t)=\pm\sqrt{(n-1)m}t$. It is easy to see that $\nabla
f=\pm\sqrt{(n-1)m}\partial_{t}$, hence $Hess\,f=0.$ Therefore
$(\mathbb{H}^{n}\times\mathbb{R},\,g,\,\nabla f,\,-(n-1))$ is a
quasi-Einstein metric.
\end{example}
Next we shall describe our second example, where its
associated vector field is not a Killing vector field.
\begin{example}
\label{ex2} We consider $\mathbb{H}^{n}\times\mathbb{R}$ with
standard metric (\ref{metHR}) and the potential function
$f(x,t)=-m\ln(\cosh(\eta t+a)),$ where $a\in\mathbb{R}$ and
$\eta=\sqrt{\frac{n-1}{m}},$ hence $\nabla f=-m\eta\tanh(\eta
t+a)\partial_{t}.$ Under these conditions
$(\mathbb{H}^{n}\times\mathbb{R},\,g,\,\nabla f,\, -(n-1))$ is a
quasi-Einstein metric.
\end{example}

For  the sake of completeness it is important to highlight that a
basic object of study in general relativity is a Lorentzian manifold
$(M^4,\,g)$ satisfying Einstein's equation
$$Ric-\frac{1}{2}Rg=8\pi T,$$ where $R$ and $T$ stand, respectively,
for the scalar curvature and the stress-energy tensor of matter. The
first solution of the Einstein equation (with $T=0$) was obtained by
Schwarzchild in $1916$, for more information about this subject we
recommend \cite{Corvino}. About this issue it is important to recall
that a static space-time is a four-dimensional manifold which
possesses a time-like vector field and a spacelike hypersurface
which is orthogonal to the integral curves of this Killing field,
see \cite{wald}. Static space-times are  special and important
global solution to Einstein equation in general relativity. Note that $1$-quasi-Einstein metrics satisfying
$\Delta e^{-f}+\lambda e^{-f}=0$ are \emph{static metrics} with
cosmological constant $\lambda.$ These static metrics have been
studied extensively because their connection with scalar curvature,
the positive mass theorem and general relativity, for more details
see e.g. \cite{Anderson}, \cite{AndersonKhuri} and \cite{Corvino}.
On the other hand, we recall that for a Riemannian manifold
$(M^n,\,g)$ the linearization $\mathfrak{L}_{g}$ of the scalar
curvature operator is given by
$$\mathfrak{L}_{g}(h)=-\Delta_{g}(tr_{g}(h))+div(div(h))-g(h,Ric_{g}),$$ where
$h$ is a $2$-tensor. Moreover, the formal $L^{2}$-adjoint
$\mathfrak{L}_{g}^{*}$ of $\mathfrak{L}_{g}$ is given by
\begin{equation}
\label{eqlinear}
\mathfrak{L}_{g}^{*}(u)=-(\Delta_{g}u)g+Hess\,u-uRic_{g},
\end{equation}
where $u$ is a smooth function on $M^n.$

With these definitions $u$ is a nontrivial element in the kernel of
$\mathfrak{L}_{g}^{*}$ if and only if the warped product metric
$\overline{g}=-u^{2}dt^{2}+g$ is Einstein, for more details see
Proposition 2.7 in \cite{Corvino}.

Since $(M^n,\,g,\,\nabla f,\lambda)$ is a quasi-Einstein metric with
$m<\infty$ we may consider $u=e^{-\frac{f}{m}}$ to rewrite the
fundamental equation (\ref{eqqem}) as
\begin{equation}
\label{equ}
 Ric-\frac{m}{u}Hess\,u=\lambda g.
\end{equation}
Therefore, combining (\ref{eqlinear}) and (\ref{equ}) we may
conclude from Examples \ref{ex1} and \ref{ex2} that
$\big(\mathbb{H}^{n}\times\mathbb{R},g\big)$ produce Einstein warped products.

Now, we are in position to announce our main result. It will show
that the above examples are unique. More precisely, we have the
following result.

\begin{theorem}
\label{thmA} Let $\big(\mathbb{H}^{n}\times\mathbb{R},\,g,\,\nabla
f,\,\lambda)$ be a quasi-Einstein metric. Then, up to an additive constant, this structure is
given according either Example \ref{ex1} or Example \ref{ex2}.
\end{theorem}

\begin{remark}
We point out that Theorem \ref{thmA} shows that there are two 
natural Einstein warped product structures with base $\mathbb{H}^{n}\times\mathbb{R}$
for every $m$ with warped function given by $e^{-\frac{f}{m}},$
where $f$ is such as in Examples \ref{ex1} and \ref{ex2}. Moreover,
we recall that Riemannian metrics with $Ker\,\mathfrak{L_{g}^{*}}$
nontrivial are called \emph{static}. Therefore, taking into account
$m=1,$ it is easy to check that $\Delta e^{-f}+\lambda e^{-f}=0,$
which permits to conclude from (\ref{eqlinear}) that
$\mathbb{H}^{n}\times\mathbb{R}$ endowed with such metric is a static metric.
\end{remark}

\begin{remark}
We highlight  that, mutatis mutandis, Theorem \ref{thmA} can be obtained to $\Bbb{E}^{n}\times\Bbb{R},$ where $\Bbb{E}^{n}$ is a Einstein manifold with negative scalar curvature.
\end{remark}

\section{Preliminaries}

Throughout this section we collect a couple of lemmas that will be
useful in the proofs of our results. Our object of study is
$\Bbb{H}^n\times \Bbb{R}$ given by
 $$\{(x,t)\in \Bbb{R}^n\times
 \Bbb{R};\,x_{n}>0,\,\hbox{where}\,x=(x_{1},...,x_{n})\}$$ endowed
 with metric
 \begin{eqnarray*}
 g=\frac{1}{x_{n}^{2}}\sum_{i=1}^{n}dx_{i}^{2}+dt^{2}.
 \end{eqnarray*}

It is easy to see that
 $\{E_{i}=x_{n}\partial_{x_{i}},\,E_{n+1}=\partial_{t}\}$ with $i$ ranging from $1$ to
 $n,$ gives a global orthonormal frame. Moreover, the Lie brackets in $\Bbb{H}^n\times
 \Bbb{R}$ satisfies
$$[E_{l},E_{n}]=-E_{l},\,\,\hbox{when}\,\,l=1,...,n-1$$
and $$[E_{j},E_{k}]=0,$$ otherwise.

Therefore, we may use the Koszul's formula to obtain the following
Riemannian connections $\nabla_{E_{l}}E_{l}=E_{n}$ and
$\nabla_{E_{l}}E_{n}=-E_{l}$ with $l$ ranging from $1$ to $n-1$ and
$\nabla_{E_{j}}E_{k}=0$ in the other cases.

Now, we may use the Riemannian connection to deduce the following
lemma.

\begin{lemma}
\label{prop1} The Ricci tensor of $\Bbb{H}^n\times \Bbb{R}$ is given
by
\begin{equation*}
Ric=-(n-1)g+(n-1)dt^{2}.
\end{equation*}
\end{lemma}
\begin{proof}
Firstly, we recall that $Ric(X,Y)=\sum_{k=1}^{n+1} \langle
R(E_{k},X)Y,E_{k}\rangle,$ where $\{E_{1},...,E_{n+1}\}$ is an
orthonormal frame. From what it follows that for $l=1,...,n-1$ we have

\begin{eqnarray*}
Ric(E_{l},E_{l}) &=& \sum_{k\neq l}\langle \nabla_{E_{k}}\nabla_{E_{l}}E_{l}-\nabla_{E_{l}}\nabla_{E_{k}}E_{l}-\nabla_{[E_{k},E_{l}]}E_{l},E_{k}\rangle \\
                 &=& \langle\nabla_{E_{n}}E_{n}-\nabla_{[E_{n},E_{l}]}E_{l},E_{n}\rangle+\sum_{k\neq l,n}\langle \nabla_{E_{k}}E_{n}-\nabla_{[E_{k},E_{l}]}E_{l},E_{k}\rangle \\
                 &=& \langle-\nabla_{E_{l}}E_{l},E_{n}\rangle+\sum_{k\neq l,n}\langle \nabla_{E_{k}}E_{n},E_{k}\rangle \\
                 &=& -1+\sum_{k\neq l,k<n}\langle-E_{k},E_{k}\rangle \\
                 &=& -(n-1).
\end{eqnarray*}
In a similar way a straightforward computation gives
$Ric(E_{n},E_{n})=-(n-1),$ $Ric(E_{n+1},E_{n+1})=0$ and
$Ric(E_{i},E_{j})=0$ if $i\neq j,$ we left its checking  for the
reader. So, we finishes the proof of the lemma.
\end{proof}

Now, we consider that $\Bbb{H}^n\times \Bbb{R}$ admits a
quasi-Einstein structure and we use Lemma \ref{prop1} to
deduce the following lemma.
\begin{lemma}
\label{prop2}
 Let $\big(\mathbb{H}^{n}\times\mathbb{R},\,g,\,\nabla
f,\,\lambda)$ be a quasi-Einstein structure. Then the following
statements hold:
\begin{enumerate}
\item  $x_{n}^{2}\frac{\partial^{2}f}{\partial x_{i}^{2}}-x_{n}\frac{\partial f}{\partial x_{n}}=\frac{1}{m}x_{n}^{2}\left(\frac{\partial f}{\partial x_{i}}\right)^{2}+\lambda+(n-1)$, if $i<n$;
\item  $x_{n}\frac{\partial f}{\partial x_{n}}+x_{n}^{2}\frac{\partial^{2} f}{\partial x_{n}^{2}}=\frac{1}{m}x_{n}^{2}\left(\frac{\partial f}{\partial x_{n}}\right)^{2}+\lambda+(n-1)$;
\item  $\frac{\partial^{2}f}{\partial t^{2}}=\frac{1}{m}\left(\frac{\partial f}{\partial t}\right)^{2}+\lambda$;
\item  $\frac{\partial^{2}f}{\partial x_{i}\partial x_{j}}=\frac{1}{m}\frac{\partial f}{\partial x_{i}}\frac{\partial f}{\partial x_{j}}$, if $i<j<n$;
\item  $x_{n}^{2}\frac{\partial^{2}f}{\partial x_{i}\partial x_{n}}+x_{n}\frac{\partial f}{\partial x_{i}}=\frac{1}{m}x_{n}^{2}\frac{\partial f}{\partial x_{i}}\frac{\partial f}{\partial x_{n}}$, if $i<n$;
\item  $\frac{\partial^{2}f}{\partial x_{i}\partial t}=\frac{1}{m}\frac{\partial f}{\partial x_{i}}\frac{\partial f}{\partial t}$, if $i\leq n$.
\end{enumerate}
\end{lemma}
\begin{proof}
In order to obtain the first item we may compute
$Ric_{f}^{m}(E_{i},E_{i})$ with $i<n$ in  equation
(\ref{eqqem}) to obtain
\begin{eqnarray*}
Ric(E_{i},E_{i})+Hess\,f(E_{i},E_{i})-\frac{1}{m}df\otimes
df(E_{i},E_{i})=\lambda.
\end{eqnarray*}
Since $E_{i}=x_{n}\partial_{x_{i}}$ we can apply Lemma
\ref{prop1} to obtain the first assertion.

Proceeding, we compute $Ric_{f}^{m}(E_{n},E_{n})$ and once more we
use Lemma \ref{prop1} to infer
\begin{eqnarray*}
-(n-1)+(n-1)\langle
E_{n},\partial_{t}\rangle^{2}+Hess\,f(E_{n},E_{n})-\frac{1}{m}df\otimes
df(E_{n},E_{n})=\lambda.
\end{eqnarray*}
Thus, we use that $E_{n}=x_{n}\partial_{x_{n}}$ to deduce the second
statement.

Analogously, computing $Ric_{f}^{m}(E_{n+1},E_{n+1}),$
$Ric_{f}^{m}(E_{i},E_{j})$ with $i<j<n,$ $Ric_{f}^{m}(E_{i},E_{n})$
with $i<n$ and $Ric_{f}^{m}(E_{i},E_{n+1})$ with $i\leq n$
straightforward computations give the desired statements, which
complete the proof of the lemma.
\end{proof}

As an application of Lemma \ref{prop2}, finally we have the following
lemma.

\begin{lemma}
\label{lem1} Let $\big(\mathbb{H}^{n}\times\mathbb{R},\,g,\,\nabla
f,\,\lambda)$ be a quasi-Einstein structure. Then either
$\frac{\partial f}{\partial t}(x,t)=\pm\sqrt{-m\lambda}$ or
$\frac{\partial f}{\partial t}(x,t)=-\sqrt{-m\lambda}\tanh(\mu
t+a),$ where $a=a(x)$ and $\mu=\sqrt{-\frac{\lambda}{m}}$.
\end{lemma}
\begin{proof}
First,  we define $h(t)=\frac{\partial f}{\partial t}(x,t)$ for a fixed $x\in\mathbb{H}^{n}\times\mathbb{R}.$
Therefore, we use the third item of Lemma \ref{prop2} to write
\begin{equation}
\label{eqedo}
h'=\frac{h^{2}}{m}+\lambda.
\end{equation}
If $h' \equiv 0,$ then $h=\pm\sqrt{-m\lambda},$ which gives the
first assertion.

On the other hand, supposing $h^{'}\neq 0,$ then (\ref{eqedo})
can be rewrite as $$h'=\frac{1}{(m^{-1}h^{2}+\lambda)^{-1}}$$ and
thus it is a separable ODE. Therefore, its solutions are given by

$$h(t) = \sqrt{m\lambda}\tan\left(\sqrt{\frac{\lambda}{m}}(t+c)\right)\,\,\hbox{if}\,\,\lambda>0,$$
$$h(t) = -\frac{m}{t+mc},\,\,\hbox{if}\,\,\lambda=0$$
and
$$\left|\frac{h(t)-\sqrt{-m\lambda}}{h(t)+\sqrt{-m\lambda}}\right| =
\exp{\left(2\sqrt{-\frac{\lambda}{m}}(t+c)\right)},\,\,\hbox{if}\,\,\lambda<0,$$
where $c$ is a real constant. Now, we may delete the two first
solution by using the differentiability of $f.$ Therefore, it remains
only the last solution. In this case we have two possibilities
\begin{equation}
\label{eq1plmk}
\frac{h(t)-\sqrt{-m\lambda}}{h(t)+\sqrt{-m\lambda}}=\exp{\left(2\sqrt{-\frac{\lambda}{m}}(t+c)\right)}
\end{equation}
and
\begin{equation}
\label{eq2plmk}
\frac{h(t)-\sqrt{-m\lambda}}{h(t)+\sqrt{-m\lambda}}=-\exp{\left(2\sqrt{-\frac{\lambda}{m}}(t+c)\right)}.
\end{equation}

Taking into account (\ref{eq1plmk}) we consider
$\mu=\sqrt{-\frac{\lambda}{m}}$ to arrive at
\begin{eqnarray*}
h(t) &=& \sqrt{-m\lambda}\coth(-\mu(t+c)) ,
\end{eqnarray*}
which gives a contradiction with the differentiability of $h.$
Therefore, from (\ref{eq2plmk}) we conclude that
\begin{eqnarray*}
h(t) = \sqrt{-m\lambda}\tanh(-\mu(t+c)) = -\sqrt{-m\lambda}\tanh(\mu
t+a).
\end{eqnarray*}
This is what we wanted to prove.
\end{proof}

We point out that in Lemma \ref{lem1} the constant $\lambda$ must be negative; for more details see Propositon 3.6 in \cite{csw}, see also  \cite{case}.

\section{The Proof of Theorem \ref{thmA}}
\begin{proof}
First, since $\Bbb{H}^n\times \Bbb{R}$ admits a quasi-Einstein
structure, we obtain from Lemma \ref{lem1} two possibilities given
by
$$\frac{\partial f}{\partial t}(x,t)=\pm\sqrt{-m\lambda}$$
and $$\frac{\partial f}{\partial t}(x,t)=-\sqrt{-m\lambda}\tanh(\mu
t+a),$$ where $a=a(x)$ and $\mu=\sqrt{-\frac{\lambda}{m}}$.

In the first case we may use item $(6)$ of Lemma \ref{prop2} to
arrive at $\frac{\partial}{\partial x_{i}}\left(\frac{\partial
f}{\partial t}\right)=\frac{1}{m}\frac{\partial f}{\partial
x_{i}}\frac{\partial f}{\partial t}$. Hence
$\frac{1}{m}\frac{\partial f}{\partial
x_{i}}(\pm\sqrt{-m\lambda})=0.$ From what it follows that
$\frac{\partial f}{\partial x_{i}}=0$ for $i\leq n$, where we used
that $\lambda<0.$  Moreover, we use item $(1)$ of Lemma
\ref{prop2} to conclude $\lambda=-(n-1)$ and
$f(x,t)=\pm\sqrt{m(n-1)}t+c,$ $c$ constant. So, we prove the first part.

Proceeding we consider $\frac{\partial f}{\partial
t}(x,t)=-\sqrt{-m\lambda}\tanh(\mu t+a).$ Once more, we use item $(6)$
of Lemma \ref{prop2} to obtain

$$\frac{\partial}{\partial
x_{i}}\left(\frac{\partial f}{\partial
t}\right)=\frac{1}{m}\frac{\partial f}{\partial x_{i}}\frac{\partial
f}{\partial t},$$ which implies that
\begin{equation}
\label{eq1pthm1} sech^{2}(\mu t+a)\frac{\partial a}{\partial
x_{i}}=\frac{1}{m}\tanh(\mu t+a)\frac{\partial f}{\partial x_{i}},
\end{equation} for every $(x,t)\in \Bbb{H}^{n}\times \Bbb{R}.$

On the other hand, fixing $x$ and choosing $t$ such that $\tanh(\mu
t+a(x))=0$ we have from (\ref{eq1pthm1}) that $sech^{2}(\mu
t+a(x))\frac{\partial a}{\partial x_{i}}(x)=0.$ Since $sech^{2}(\mu
t+a(x))$ can not assume null value we conclude that $\frac{\partial
a}{\partial x_{i}}(x)=0$. Thus, since $x$ is arbitrary we obtain
$\frac{\partial a}{\partial x_{i}}\equiv 0,$ which implies that $a$
is constant. Therefore, we have $\frac{1}{m}\tanh(\mu
t+a)\frac{\partial f}{\partial x_{i}}=0$ for every
$(x,t)\in\Bbb{H}^n\times \Bbb{R}$. From what it follows that
$\frac{\partial f}{\partial x_{i}}=0$, for $i\leq n$. By using again
the first item of Lemma \ref{prop2} we arrive at $\lambda=-(n-1).$
So, we have finish the proof of the theorem.
\end{proof}

\end{document}